\documentclass[11pt]{amsart}
\usepackage {amsmath, amssymb, a4wide, epsfig, enumerate, psfrag}
\usepackage[latin1]{inputenc}
\usepackage[english]{babel}

\date{\today}

 \author{Romain Dujardin}
 \thanks{Research partially supported by ANR project BERKO}
\address{Institut de mathématiques de Jussieu, UMR 7586 du CNRS, Projet Géométrie et Dynamique,
 \& UFR de mathématiques Université Paris 7, Site Chevaleret, Case 7012, 75205 Paris cedex 13, 
         France}
\email{dujardin@math.jussieu.fr}
\title{Wermer examples and currents}

\newcommand{\cc}{\mathbb{C}}

\newcommand{\re}{\mathbb{R}}
\newcommand{\dd}{\mathbb{D}}
\newcommand{\zz}{\mathbb{Z}}
\newcommand{\nn}{\mathbb{N}}
\newcommand{\e}{\varepsilon}
\newcommand{\cv}{\rightarrow}

\newcommand{\fr}{\partial}
\newcommand{\om}{\Omega}
\newcommand{\set}[1]{\left\{#1\right\}}
\newcommand{\norm}[1]{\left\Vert#1\right\Vert}
\newcommand{\abs}[1]{\left\vert#1\right\vert}
\newcommand{\cd}{{\cc^2}}

\newcommand{\rest}[1]{ \arrowvert_{#1}}

\newcommand{\unsur}[1]{\frac{1}{#1}}

\newcommand{\rond}{\hspace{-.1em}\circ\hspace{-.1em}}
\newcommand{\cS}{\mathcal{S}}
\newcommand{\lrd}{\underline{\mathrm{rad}}}
\newcommand{\urd}{\overline{\mathrm{rad}}}

\newcommand{\itm}{\item[-]}

\DeclareMathOperator{\supp}{Supp}

\DeclareMathOperator{\dist}{dist}

\newtheorem{prop} {Proposition} [section]
\newtheorem{thm}[prop] {Theorem} 

\newtheorem{lem}[prop] {Lemma}
\newtheorem{cor}[prop]{Corollary}
\newtheorem{theo}{Theorem} 

\theoremstyle{remark}

\newtheorem{rmk}[prop]{Remark}

\begin{document}
\maketitle

\begin{abstract}
We give the first examples of positive closed currents $T$ in $\mathbb{C}^2$
 with continuous potentials, $T\wedge T=0$  and whose supports do not contain any holomorphic disk. 
This gives in particular an affirmative answer to a question of Forn\ae ss and Levenberg.
We actually construct examples with potential
 of class $C^{1,\alpha}$ for all $\alpha<1$. This regularity is expected to be essentially optimal.
\end{abstract}

\section*{Introduction}

The purpose of this paper is to investigate the geometric properties 
 of positive closed currents $T$ with $T\wedge T=0$ in  complex dimension two. 
Equivalently, we deal with plurisubharmonic functions $u$, 
solutions to the homogeneous Monge-Ampère equation $(dd^cu)^2=0$.
 
Let us start with some generalities about geometric currents. A $(1,1)$ positive closed current in $\cd$
is  {\em uniformly laminar} if it is an integral 
of currents of integration over the leaves of a lamination. A current  is said 
to be {\em laminar} if it is uniformly laminar outside a set of arbitrary small trace measure. 
These concepts have been introduced in \cite{bls} and have  proved to be 
very useful since then. 
If $T$ is uniformly laminar, 
then $T\wedge T=0$ as soon as the wedge product makes sense \cite{isect}. 
Likewise, any local plurisubharmonic (psh for short) potential for $T$ is harmonic along the leaves of the underlying lamination.

\medskip

Back to our initial problem, we work locally  in  $\cd$ so let us consider a psh potential $u$ for $T$. We assume that $u$ is, say, bounded, so that the self-intersection $T\wedge T = (dd^c u)^2$ is well defined, and vanishes. If $u$ is of class $C^3$, the Frobenius Integrability Theorem implies that there exists a foliation by holomorphic disks along which $u$ is harmonic (see \cite{bk}), and $T$ is a uniformly laminar current associated to this foliation.
It is expected, but apparently still unknown, that this result should carry over for $u$ of class up to $C^2$ (see \cite{bedford}). As we shall demonstrate here, the situation is dramatically different for regularity below 
$C^{2}$. Before entering into the details of our results, let us mention that the laminarity properties of the solutions to homogeneous Monge-Ampère equations have recently played a prominent role in connection with the study of  extremal metrics in K\"ahler
  geometry \cite{do, chen-tian}. The natural regularity appearing in this setting is $C^{1,1}$.

\medskip

It is a classical result due to Sibony 
that there exist examples of currents $T$ with $T\wedge T=0$ and $C^{1,1}$ potential,
 which are not uniformly laminar (this construction was reported  e.g. in \cite{bf, fornaess-levenberg}). Indeed let $\mathbb{B}$
 be the unit ball in $\cd$ and $X \subset \fr \mathbb{B}$ be a closed set with the property that the polynomial hull $\widehat X$ does not contain
 any holomorphic disk (a so-called Stolzenberg or Wermer example, see below for more details). Let $f\in C^\infty(\fr \mathbb{B})$
 be a nonnegative function such that $X=\set{f=0}$ and let $u$ be the unique psh function in $\mathbb{B}$, continuous in  $\overline{\mathbb{B}}$
 such that  $u\rest{\fr \mathbb{B}} =f$ and $(dd^cu)^2=0$ \cite{bt}. Let also $T=dd^cu$.  Then  $u$ is  of class $C^{1,1}$, nonnegative, and 
$\set{u=0} =\widehat X$. Now if $p\in \widehat X$ and  $\Delta$ is any holomorphic disk through $p$, then $u$ cannot be harmonic along $\Delta$. Indeed, $\Delta$ is not contained in $\widehat X$ so $u\rest{\Delta}$ is not identically 0, and $u$ has a minimum at $p$ so it is not harmonic. This both 
shows that $p \in \supp(T)$ and that $T$  cannot be uniformly laminar near $p$. 

On the other hand it can be shown (see Proposition \ref{prop:sibony} below) that in this situation $\widehat X$ has
always  zero trace measure  (relative to $T$) so nothing prevents these currents from being uniformly laminar  on an open set of full mass.

To the best of our knowledge, an example is still lacking of a non-laminar current $T$ with $T\wedge T=0$, even with merely bounded potential.
We fill this gap by proving the following result.

\begin{theo}\label{thm:wermer}
 There exists a  closed positive $(1,1)$ current $T = dd^cu$ in the unit bidisk  $\dd^2\subset\cd$ such that:
\begin{enumerate}[i.]
 \item $u$ is of class $C^{1,\alpha}$ for all $0<\alpha<1$;
 \item $T\wedge T=(dd^c u)^2=0$;
 \item the support of $T$ does not contain any holomorphic disk.
\end{enumerate}
\end{theo}

Recall that a function is of class $C^{1,\alpha}$ if it is differentiable and its derivatives are Hölder continuous of exponent $\alpha$. Likewise,  given any continuous increasing function $\psi$ with $\psi(0)=0$,  we say that $u$ is $C^{1+\psi}$ if its derivatives have modulus of continuity $O(\psi)$. Our method actually produces examples of potentials $u$ with regularity $C^{1+\psi}$, where $\psi$ is any 
modulus of continuity such that  $\frac{\psi(\delta)}{ \delta \abs{\log \delta}}\cv \infty$  --see Theorem \ref{thm:refined} for a precise statement. 
On the other hand it is  a feature  of our construction that these examples cannot be made $C^{1,1}$ (see \S \ref{subs:hausdorff}). 

Observe that item {\em iii.} of Theorem \ref{thm:wermer} cannot  be true when $u$ is $C^2$, for $\supp(T)$ would have nonempty interior in this case.  In \cite{bedford}, Bedford asks whether a foliation exists on a dense subset of $\supp(T)$ when $u$ is $C^{1,1}$. 

\medskip

By definition, a Wermer example is a subset in the unit bidisk $\dd^2$ which is the polynomial hull of a
compact subset of $\fr \dd\times\dd$,  projects onto $\dd$ under the first projection and contains no holomorphic disk.
The existence of such objects is originally due to  Wermer \cite{we} and they have subsequently 
been studied by several authors \cite{lev, alexander-michigan, slod, duval-s}.  
 
It is an easy observation of \cite{duval-s} that there exist  Wermer examples supporting 
positive closed currents, therefore  currents whose support do not contain any holomorphic disk. This observation is
 the starting point of Theorem \ref{thm:wermer}.

However, a source of difficulty is that the Wermer examples constructed thus far tend to be ``small'', 
and it is delicate to estimate their size from below anyway. For instance, in 
\cite{fornaess-levenberg}, the authors ask whether there exists a non pluripolar Wermer example. 
Of course, pluripolarity of the support of a current is an obstruction to the local boundedness of its potential. 

What we do here is first to provide a construction of ``thick'' Wermer examples (Section \ref{sec:wermer}). 
 Then we develop methods to get effective lower bounds on the size of such objects, in the pluripotential-theoretic sense 
(Sections \ref{sec:continuous}  and \ref{sec:regularity}). We ultimately take advantage of the regularity properties 
of the solutions to the  homogeneous Monge-Ampère equation (\S \ref{subs:mongeampere}).

In particular, Theorem \ref{thm:wermer} in the bounded potential case 
gives a positive answer to the question of Forn\ae ss and Levenberg. 

Another question is to determine what the Hausdorff dimension of a Wermer example can be. 
The examples we construct have dimension up to 4, 
but probably always zero Lebesgue measure (see \S \ref{subs:hausdorff}). 

A related issue is the Stolzenberg ``swiss cheese'' example 
\cite{stolzenberg}, where no boundary condition is imposed. 
Stolzenberg-like examples with positive Lebesgue measure 
have been constructed in \cite{duval-levenberg}, 
nevertheless I don't know how to use them to obtain interesting currents.

\section{Wermer examples}\label{sec:wermer}

In this section we provide a construction of Wermer examples, 
based on  that of \cite{duval-s}.
We actually arrange so that our objects have some laminar structure
 near the boundary of the bidisk, which will be useful for  regularity issues. 
 Therefore to get an actual  Wermer example  it will be enough to restrict to a  smaller bidisk.

We denote by $D(a,r)$ the disk of center $a$ and radius $r$ in $\cc$, and write  $\dd = D(0,1)$.
We say that a  subset $X$ in $\dd\times\dd$ is {\em horizontal}  if $ X\subset \dd\times D(0,1-\e)$ for some $\e>0$. 
A current is horizontal if its support is.
Dividing the $z$ coordinate by 2 we work the bidisk $D(0,1/2)\times\dd$ --this is convenient  for  if $z,z'\in D(0,1/2)$, then
 $\abs{z-z'}<1$. 

\medskip

Let first  $(a_n)_{n\geq 1}$ be a sequence of  points in $D(0,1/4)$ such that  $(a_{2p})$ and  $(a_{2p+1})$
are dense in that disk. We put $A_n(z,w)= z-a_n$ if $n$ is odd and $z+\frac{w}{100}-a_n$ if $n$ is even. Note that  
$\abs{A_n}\leq 1$ in $D(0, 1/2)\times\dd$.

We will inductively define families of polynomials $P_{n,s}$, where $n\in \nn$ and $s$ ranges through a finite set 
$\mathcal{S}_n$. Fix $P_0(z,w)=w$, and $\mathcal{S}_0 = \set{0}$. 

Let $(\delta_n)_{n\geq 0}$ and $(\e_n)_{n\geq 1}$ be  sequences of positive real numbers, with $\delta_0 = 1/2$, and $(m_n)_{n\geq 1}$ 
be a sequence of positive integers. The inductive step is as follows. Assume that
$\mathcal S_n$ and the polynomials $(P_{n,s})_{s\in \mathcal{S}_n}$ have been constructed, 
and consider the finite set $\Sigma_{n+1}:= \overline{D}\big(0, \delta_n(1-\unsur{m_{n+1}})\big)\cap \frac{3\delta_n}{m_{n+1}} \zz^2$. 
That is, $\Sigma_{n+1}$ is the set of those $\sigma \in D(0, \delta_n)\cap \frac{3\delta_n}{m_{n+1}} \zz^2$ such that
$D(\sigma, \frac{\delta_n}{m_{n+1}})\subset D(0, \delta_n)$.  By construction,
the disks $D(\sigma, \frac{\delta_n}{m_{n+1}})$ have disjoint closures.

For large $m_{n+1}$, $\#\Sigma_{n+1} \sim \frac{\pi}{9} m_{n+1}^2$. On the contrary, observe that when $m_{n+1} = 1$, 
$\Sigma_{n+1} = \set{0}$. 

Let $\cS_{n+1} = \cS_n\times \Sigma_{n+1}$ and for $s' = (s,\sigma)\in \cS_{n+1}$ let 
\begin{equation}
P_{n+1, s'} = (P_{n,s} - \sigma)^2 - \e_{n+1} A_{n+1}.
\end{equation}

Put $X_n = \bigcup_{s\in \mathcal S_n} \set{\abs{P_{n,s}}<\delta_n}$.
It is useful to think about the inductive definition of $X_n$  as being made up  of two steps: we  first replace 
$X_{n,s} := \set{\abs{P_{n, s}}<\delta_n}$ by $\bigcup_{\sigma\in \Sigma_{n+1}} X^{\mathrm{int}}_{n+1,s,\sigma}$ , where 
$X^{\mathrm{int}}_{n+1,s,\sigma}:=
\set{\abs{P_{n, s}-\sigma}<
\frac{\delta_n}{m_{n+1} }}$
 (``subdivision''), and  then the intermediate 
$X^{\mathrm{int}}_{n+1,s,\sigma}$ with $X_{n+1, s'}=
\set{\abs{P_{n+1, s'}}<\delta_{n+1}}$ (``ramification''). 
As compared to \cite{we,lev, slod, duval-s}, the subdivision  step is new.

\begin{lem}\label{lem:construction}
 Fix a  sequence of positive real numbers $(r_n)_{n\geq 1}$,  decreasing to zero, with $r_n\leq \unsur{10}$ 
Let  $(\delta_n)_{n\geq 0}$ be the sequence defined by $\delta_0 =1/2$ and 
$\delta_{n+1} = \frac{\delta_n^{2}r_{n+1}}{4 m_{n+1}^2}$ 
and $(\e_n)_{n\geq 1}$ be defined by $\e_{n+1} =\frac{\delta_n^{2}}{2m^2_{n+1}}$.

Then  the following properties hold for every $n\geq1$:
\begin{enumerate}[(i.)]
 \item $\overline{X_{n+1}}\subset {X_n}$ in $D(0, 1/2)\times\dd$; more precisely, with notation as above for every $s'=(s,\sigma)$,  we have that 
 $\overline{X_{n+1, s'}}\subset {X^{\mathrm{int}}_{n+1,s,\sigma}}\subset X_{n,s}$;
 \item $X_{n+1}$ does not contain the  
 graph   of any holomorphic (even merely continuous) function over $D(a_{n+1}, r_{n+1})$,
 relative to the  projection $\pi_0(z,w) = z$ if $n$ if even, relative to $\pi_1(z,w) = z + \frac{w}{100}$ if $n$ is odd;
 \item for each $s\in \cS_n$ and 
  $\alpha\in \cc$ with $\abs{\alpha}< 2\delta_n$, the analytic set
$\set{P_{n,s} = \alpha}$ is horizontal in $D(0, 1/2)\times\dd$, of degree $2^n$ and is a (non ramified) covering
 over $\set{\frac38\leq \abs{z} \leq \frac12}$, relative to  $\pi_0$. 
Furthermore, if $s_1\neq s_2$, the varieties
$\set{P_{n,s_1} = \alpha}$ and $\set{P_{n,s_2} = \alpha}$ are disjoint
\end{enumerate}
\end{lem}

\begin{proof}
It is obvious that for all $s'=(s,\sigma)\in \cS_{n+1}$,
 ${X^{\mathrm{int}}_{n+1,s,\sigma}}\subset X_{n,s}$.
Assuming that the constant $\delta_n$ has been chosen,  to ensure the inclusion
 $\overline{X_{n+1, s'}}\subset {X^{\mathrm{int}}_{n+1,s,\sigma}}$ it is enough that 
\begin{equation}\label{est1}
 \delta_{n+1} + \e_{n+1}< \frac{\delta_n^{2}}{m_{n+1}^2}.
\end{equation}
Let us also observe that since $X_0 = \set{\abs{P_0}<\delta_0=1/2}$ is horizontal, the horizontality assertion in  {\em (iv.)} follows from the fact that $X_n\subset X_0$.

We will use the following elementary lemma, which will be proved afterwards. 

\begin{lem}\label{lem:sqrt}
If $\delta<\e r$, there does not exist any continuous function  $f$ on $D(0,r)$ such that $\abs{(f(\zeta))^2-\e\zeta}<\delta$ for  $\zeta\in  D(0,r)$.
\end{lem}
%

From this we infer that to meet condition {(\em ii.)} it is enough that for every $n$, 
\begin{equation}\label{est2}
 \delta_{n+1}< \e_{n+1}r_{n+1}
\end{equation}

It is clear from the explicit definition of $(\delta_n)$ and $(\e_n)$ that \eqref{est1} and \eqref{est2}, whence {\em (i.)} and {(\em ii.)} hold. 

It remains to check {\em (iii.)} It is clear that $P_{n,s}$ has degree $2^n$ in $w$ so it is enough to prove that the equation 
$P_{n,s}(z_0, w) = \alpha$ has at least (hence exactly) $2^n$ distinct roots for each fixed $z_0$ with ${3/8<\abs{z_0}<1/2}$. Fix such a $z_0$.
We will prove by induction the following slightly stronger fact: let 
$w\mapsto \gamma(w)$ 
be a holomorphic function on  $\dd$, such that  $\abs{\gamma}< 2\delta_n$; then for every $s\in \mathcal{S}_n$, 
the equation  $P_{n,s}(z_0,w) =\gamma$  has at least $2^n$ distinct solutions in $\dd$.  For $n=0$ this follows from Rouché's Theorem.

For convenience we drop the $z_0$ and consider our functions as depending solely on $w$.
Assume the result  holds for $n$, and consider the equation 
$P_{n+1, s'}= \gamma$ where  $\abs{\gamma}<2\delta_{n+1}$ in $ \dd$, that is, $(P_{n,s}-\sigma)^{2} = \gamma +\e_{n+1}A_{n+1}$. The right hand side does not vanish on $U\times\dd$. Indeed  $\gamma +\e_{n+1}A_{n+1}= 0$ is equivalent to $A_{n+1} = -\gamma/\e_{n+1}$, and with the choices that we have made, 
\begin{equation}
 \label{eq:an}
\abs{\frac{\gamma}{\e_{n+1}}}<r_n\leq \unsur{10}\text{ while }\abs{A_{n+1}}>\unsur{8}-\unsur{100}>\unsur{10}. 
\end{equation}
In particular the function  $\gamma 
+\e_{n+1} A_{n+1} $ admits two square
roots $\pm  g$  in $U\times \dd$. We have that 
   $0<\abs{g}<(2\delta_{n+1}+ \e_{n+1})^{1/2}<\delta_n/m_{n+1}$ and 
 the equation $P_{n+1, s'} = \gamma$ is equivalent to $\set{P_{n,s} = \sigma \pm  g}$; we conclude by the induction hypothesis. 
The last assertion in  {\em (iii.)} is obvious.
\end{proof}

\begin{proof}[Proof of Lemma \ref{lem:sqrt}] By scaling, it is enough to prove the result for $\e=1$.
Fix $\zeta_0$ such that  $\abs{\zeta_0} = r$. Since  $\delta<r$, the open set
$\set{z\in \cc, \ \abs{z^2-\zeta_0}<\delta}$ has two connected components. Indeed the critical value of  $z\mapsto z^2-\zeta_0$ lies outside $D(0,\delta)$. Let $U_1(\zeta_0)$ and $U_2(\zeta_0)$ be these two components. As $\zeta_0$ turns around $\fr D(0,r)$ these components are swapped.

If now $f$ is a continuous function satisfying the assumption of the lemma, reducing $r$ slightly we may assume $f$ is continuous on $\overline{D(0,r)}$. Assume $f(r)\in U_1(r)$. By making $\zeta = re^{i\theta}$, $0\leq \theta\leq 2\pi$ wind around $\fr D(0,r)$,  we see that 
 $f(r)$ also belongs to $U_2(r)$, whence the contradiction. When $f$ is holomorphic, an alternate argument is provided by Rouché's Theorem. 
\end{proof}

\begin{prop}\label{prop:wermer}
 Let  $X_n$ be as above and  set $X= \bigcap_n X_n$. 
Then $X$ is a polynomially convex  horizontal subset in $D(0,1/2)\times\dd$, and 
$X\cap (D(0,1/5)\times \dd)$ does not contain any holomorphic  disk. 
\end{prop}

\begin{proof}
The horizontality and polynomial convexity of $X$ are obvious.  By items {\em (i.)} and {\em (ii.)} --applied to odd integers-- 
of  the previous lemma,   it is clear that $X\cap \set{\abs{z}<1/4}$ does not contain any piece of  holomorphic graph over the $z$ coordinate. So any holomorphic disk contained in  
 $X\cap \set{\abs{z}<1/5}$ must be contained in a vertical line.
 Then this would be a graph over a certain open subset of $D(0,1/4)$, relative  to  the projection $\pi_1$,
 which again is impossible, still due to {\em (ii.)}. Thus $X\cap \set{\abs{z}<1/5}$ contains no holomorphic disk.
\end{proof}

Shortly we shall see  that $X$ carries a natural positive closed current $T$. What we do in the next sections  is
 to choose  the parameters carefully so that $T$ is as regular as possible. Notice that with our presentation
the free parameters are the sequences $(r_n)$ and $(m_n)$.

\section{A current  with continuous potential on $X$}\label{sec:continuous}

In this section we construct the currents $T$ associated to our Wermer examples and give their first properties. 
 The  precise regularity statement leading to Theorem \ref{thm:wermer}
  will be proven afterwards. 
Of course, to get the actual statement of  Theorem \ref{thm:wermer} 
 it is enough to restrict the conclusions of the foregoing results  to $D(0,1/5)\times\dd$ and rescale. 
We refer the reader to \cite{dem, klimek} for background on positive closed currents and psh functions.

\begin{thm}\label{thm:wermer_continuous}
Let the polynomials $(P_{n,s})_{s\in \mathcal{S}_n, \ n\in \nn}$ 
and the sequence  $(\delta_n)$ be defined as in the previous section. Consider the sequence of psh functions 
$$u_n= \unsur{2^n\#\mathcal{S}_n} \sum_{s\in \mathcal{S}_n}  \log\max(\abs{P_{n,s}}, \delta_n).$$ Then, if 
\begin{equation}\label{eq:rn}
 \sum_{n=1}^\infty \frac{\abs{\log r_n}} {2^n} <\infty,
\end{equation}
 the sequence of currents $T_n = dd^cu_n$ converges to a horizontal positive closed current  $T$ such that 
\begin{itemize}
 \itm $T$ has continuous potential and $T\wedge T=0$;
 \itm $\supp (T) \cap  (D(0,1/5)\times\dd)$ does not contain any holomorphic disk;
 \itm $T$ is uniformly laminar in $((D(0,1/2)\setminus D(0,3/8))\times \dd$.
\end{itemize}
\end{thm}

Notice that the result does not depend on $(m_n)$ so it holds for the ordinary (i.e. without subdivision) Wermer construction. 
See below \S \ref{subs:ordinary} for some comments on the regularity in this case.

\begin{proof}
 Recall the notation $X_n = \bigcup_{s\in\mathcal{S}_n} \set{\abs{P_{n,s}}<\delta_n}$ and $X= \bigcap X_n$. 
It is clear that $T_n$ is a sequence of currents with locally uniformly bounded masses, and for the moment we 
let $T$ be a cluster value of this sequence. Since 
 $\supp(T_n)$ is contained in  $\fr X_n$, $T$ has support in $X$, hence  
$\supp (T) \cap  (D(0,1/5)\times\dd)$ does not contain any holomorphic disk.

 It follows from the well known 
formula $\log^+\abs{x} = \int \log\abs{x-e^{i\theta}} d\theta$ that we have the integral representation 
$$\log\max(\abs{P_{n,s}}, \delta_n) = \int_{\re/2\pi\zz} \log\abs{P_{n,s} - \delta_n e^{i\theta}} d\theta, $$ whence $$
T_n = \unsur{2^n\#\mathcal{S}_n} \sum_{s\in \mathcal{S}_n}  \int_{\re/2\pi\zz}\left[P_{n,s} = \delta_n e^{i\theta}\right]d\theta,$$ 
From Lemma \ref{lem:construction}{(iii.)} we know that  the varieties $P_n = \delta_n e^{i\theta}$ 
are graphs over $\set{\frac38<\abs{z}<\frac12}$. It is classical 
 (see e.g. \cite{bls}) that in this situation the laminar structure passes to the limit, thus $T$ is uniformly laminar in 
$((D(0,1/2)\setminus D(0,3/8))\times \dd$.

Assume for the moment that $(u_n)$ converges uniformly, and let us see why $T\wedge T= 0$. Indeed for every $n$, $\supp(T_n) = \fr X_n$ while 
$\supp(T)\subset X$, hence  from $\fr X_n\cap X=\emptyset$, we get that  $T_n\wedge T = 0$. By uniform convergence of the potentials, we conclude that $T\wedge T=0$ (another argument is that $T_n\wedge T_n=0$ for all $n$).

\medskip

The main step is therefore to prove that  $(u_n)$ converges uniformly. The following lemma will be required (see below for the proof). 

\begin{lem}\label{lem:circles}
  For $k\leq 2$, let $\Sigma^{(k)} = \frac3k\zz^2\cap \overline{D}(0,1-\frac1k)$ and $$\nu_k =
 \unsur{\#\Sigma^{(k)}}\sum_{\sigma \in  \Sigma^{(k)}}\frac{k}{2\pi}\left[\fr D\left(\sigma, \unsur{k}\right)\right],
$$ where
 $[\fr D(\sigma, \unsur{k})]$ denotes 1-dimensional Hausdorff measure on   $\fr D(\sigma, \unsur{k})$. By convention, let $\nu_1 = \unsur{2\pi}[\fr\dd]$.
 Then  $(\nu_k)$ is a sequence of probability measures converging to the normalized Lebesgue measure on $\dd$ and having locally
uniformly bounded logarithmic potentials.
\end{lem}

To prove uniform convergence, we estimate $\abs{u_{n+1}-u_n}$. Let $v_{n+1}$ be the potential corresponding to the intermediate ``subdivision'' step. Using the  notation $s' = (s,\sigma)\in \mathcal{S}_n\times \Sigma_{n+1} = \mathcal{S}_{n+1}$ as in the previous section we have that 
$$v_{n+1} = 
\unsur{2^n \# \mathcal{S}_{n+1}} \sum_{s' = (s,\sigma)\in \mathcal{S}_{n+1}}\log\max\left(\abs{P_{n,s}-\sigma}, \frac{\delta_n}{m_{n+1}}\right).$$ Write $u_{n+1}-u_n = (u_{n+1}-v_{n+1}) + (v_{n+1}-u_n)$. 

\medskip

The second part of this equality is estimated using Lemma \ref{lem:circles} as follows:
$$v_{n+1}-u_n = \unsur{2^n\#\mathcal{S}_n} \sum_{s\in \mathcal{S}_n} 
\left( \frac{1}{\# \Sigma_{n+1}} \sum_{\sigma \in \Sigma_{n+1}} \log\max\left(\frac{\abs{P_{n,s}-\sigma}}{\delta_n}, \unsur{m_{n+1}}\right)
-  \log\max\left(\frac{\abs{P_{n,s}}}{\delta_n}, 1\right)\right).$$ Let $L_k$ be the logarithmic potential of $\nu_k$. We have that 
\begin{equation}\label{eq:vu}
 v_{n+1}-u_n  = \unsur{2^n\#\mathcal{S}_n} \sum_{s\in \mathcal{S}_n} (L_{m_{n+1}}- L_1)\rond\left(\frac{P_{n,s}}{\delta_n} \right)= O\left(\unsur{2^n}\right), 
 \end{equation}
where the second equality follows from Lemma \ref{lem:circles}.

\medskip

Now the first part writes as 
\begin{equation}\label{eq:uv}
 u_{n+1}-v_{n+1} = \unsur{2^n \# \mathcal{S}_{n+1}} \sum_{s' = (s,\sigma)\in \mathcal{S}_{n+1}}\!\!\left(
\unsur{2} \log\max\left(\abs{P_{n+1,s'}},\delta_{n+1}\right)  -
\log\max\left(\abs{P_{n,s}-\sigma}, \frac{\delta_n}{m_{n+1}}\right)\right).
\end{equation}
Let  $u_{n+1, s'} = \log\max(\abs{P_{n+1,s'}},\delta_{n+1})$ and $v_{n+1,s'} = \log\max\left(\abs{P_{n,s}-\sigma}, \frac{\delta_n}{m_{n+1}}\right)$, and recall the sets $X_{n+1,s'}$ and $X^{\mathrm{int}}_{n+1, s,\sigma}$  from Section \ref{sec:wermer}. 
We  give a uniform estimate of the quantity $\frac12 u_{n+1, s'} - v_{n+1,s'}$ 
in a vertical slice $\set{z=z_0}$.  In such a slice we have $X_{n+1,s'}\Subset X^{\mathrm{int}}_{n+1,s,\sigma} \Subset \dd$. 
Abusing notation we write $w$ for $(z_0,w)$. 

If $w\in \overline{X_{n+1,s'}}$, $v_{n+1,s'}(w) =\log \frac{\delta_n}{m_{n+1}}$ and  $u_{n+1}(w) = \log\delta_{n+1}$. Since $\delta_{n+1} = \delta_n^2r_n/4m_{n+1}^2$ we infer that 
\begin{equation}\label{eq:harmo1}
 \abs{\frac12 u_{n+1,s'}(w)-v_{n+1,s'}(w)} \leq  \abs{\log{r_n}}.
\end{equation}

If $w \notin X^{\mathrm{int}}_{n+1, s,\sigma}$, $u_{n+1,s'}(w) =  \log \abs{P_{n+1,s'}}$ and  $v_{n+1,s'}(w) =  \log\abs{P_{n,s}-\sigma}$ 
with $\abs{P_{n,s}-\sigma}\geq \frac{\delta_n}{m_{n+1}}$. Using the equality $P_{n+1,s'} = (P_{n,s}-\sigma)^2-\e_{n+1}A_{n+1}$ we infer
$$\abs{\frac12 u_{n+1,s'}-v_{n+1,s'}} = \frac12 \log \abs{1-\frac{\e_{n+1}A_{n+1}}{(P_{n,s}-\sigma)^2} }.$$
In $D(0, 1/2)\times \dd$, we have $\abs{A_{n+1}}\leq 1$ so from the definition of $\e_{n+1}$ we deduce that 
 for $w \notin X^{\mathrm{int}}_{n+1, s,\sigma}$,  
$\displaystyle{\abs{ \frac{\e_{n+1}A_{n+1}}{(P_{n,s}-\sigma)^2} }\leq \frac{1}{2}}$. We conclude that outside $X^{\mathrm{int}}_{n+1, s,\sigma}$ we have
\begin{equation}\label{eq:harmo2}
 \abs{\frac12 u_{n+1,s'}-v_{n+1,s'}} \leq \unsur{2}\log 2.
\end{equation}

In $X^{\mathrm{int}}_{n+1, s,\sigma}\setminus X_{n+1,s'}$,  $\unsur{2}u_{n+1,s'} - v_{n+1,s'}$ is harmonic and the two previous cases 
  give us a bound for this function on $\fr X^{\mathrm{int}}_{n+1, s,\sigma} \cup \fr X_{n+1,s'}$. So by \eqref{eq:harmo1},  \eqref{eq:harmo2} and 
the  maximum principle we get  that  
 $\abs{\unsur{2}u_{n+1,s'} - v_{n+1,s'}}\leq  \abs{\log {r_n}}$ there. 
 
 Summarizing the 3 cases we see that the estimate
\eqref{eq:harmo1} holds throughout $D(0,1/2)\times\dd$.

Finally, using  \eqref{eq:uv}  we conclude that $\abs{u_{n+1} - v_{n+1}} \leq \frac{\abs{\log r_n}}{2^n}$. Together with \eqref{eq:vu} this implies that 
   $\abs{u_{n+1} - u_n} = O\left(\frac{\abs{\log r_n}}{2^n}\right)$ and concludes the proof of the theorem.
\end{proof}

\begin{proof}[Proof of Lemma \ref{lem:circles}] The proof is easy so we rather sketch it.
That $(\nu_k)$ converges  to Lebesgue measure in $\dd$ is obvious so we focus on the statement on the logarithmic potentials. 
The logarithmic potential of $\nu_k$ is given by the formula
$$\unsur{\#\Sigma^{(k)}}\sum_{\sigma \in \Sigma^{(k)}}\log\max\left(\abs{z-\sigma}, \unsur{k}\right)$$
 It is enough to prove that there exists a constant $C$ such that for every $z\in \dd$ and every $r>0$, $\nu_k(D(z,r))\leq Cr$ (this of course gives more information about the convergence but we will not need it). Indeed if this estimate holds, then for $z\in \dd$
$$\abs{\int \log \abs{z-\zeta} d\nu_k(\zeta)} 
\leq \sum_{q=0}^\infty \int_{\set{2^{-q-1}\leq\abs{z-\zeta} < {2^{-q}}}}\abs{ \log \abs{z-\zeta}} d\nu_k(\zeta) \leq \log 2 + \sum_{q=0}^\infty 
C\frac{q+1}{2^q} .$$

Given such $z$ and $r$ there are three possible  
cases. Either $r\gg \unsur{k}$, say, $r\geq \frac{100}{k}$, and the number of small circles intersecting $D(z,r)$ is bounded above by $\frac{\pi}{9}r^2k^2$ up to an error of order of magnitude of $k$ times the length of $\fr D(z,r)$, that is, $O(kr)$, with $kr\leq k^2r^2/100$. Notice also that 
$\#\Sigma^{(k)}\sim \frac{k^2\pi}{9}$. In this case we conclude that $\nu_k(D(z,r))\leq 2r^2$.

The second case is when  $\frac{1}{100k}\leq r \leq \frac{100}{k}$. Then we simply argue that the number of 
small circles intersecting $D(z,r)$ is  bounded by a constant (approximately $\frac{\pi}{9}100^2$), thus 
$\nu_k(B(z,r)\leq O(1)/\#\Sigma^{(k)} = O(1/k^2) = O(r^2)$. 

The last situation is when $r\leq \frac{1}{100k}$. In this case the intersection of $D(z,r)$ with the family of small circles, if nonempty, is a piece of a small circle of length $O(r)$. We conclude that $\nu_k(D(z,r)) = O(r/k^2)$ hence $O(r)$.
\end{proof}

\section{Precise regularity of the potential}\label{sec:regularity}

Let $\psi$ be a continuous increasing function defined in a neighborhood of $0\in \re^+$, with $\psi(0)=0$.  From now on such functions will be referred to as {\em gauge functions}. We say that a function is $C^{1+\psi}$ if it is $C^1$ and its derivatives have modulus of continuity 
$O(\psi)$. Of course $C^{1,\alpha}$ regularity  corresponds  to $\psi: r\mapsto r^\alpha$.  

\medskip

In this section we prove the following  refined  version of Theorem \ref{thm:wermer}.

\begin{thm}
\label{thm:refined}
Let $(P_{n,s})$, $(m_n)$ and $(r_n)$ be as defined in Section \ref{sec:wermer}. Assume that $(r_n)$ satisfies  \eqref{eq:rn} and let 
$T$ be the current  of Theorem  \ref{thm:wermer_continuous}.

Let $\psi$ be any gauge function such   that
 $\frac{\psi(r)}{r\abs{\log r}}\cv \infty$ as $r\cv 0$. 
 
 Then it is possible to choose $(m_n)$ so that the potential of $T$ is of class 
  $C^{1+\psi}$. 
 \end{thm}

By choosing $\psi$ with  $\psi(r) = o(r^\alpha)$ for all $0<\alpha<1$ 
 (e.g. $\psi(r)= r\abs{\log r}^2$), we get  the conclusions of Theorem \ref{thm:wermer}.
It is likely that the requirement on $\psi$ could be upgraded to $\frac{\psi(r)}{r}\cv \infty$.

\subsection{Regularity of subharmonic functions}
The estimate on regularity will ultimately be a consequence of the following --presumably well known-- result. 

\begin{prop}\label{prop:modulus}
 Let $u$ be a subharmonic function in $\re^n$, with Laplacian $\Delta u = \mu$. Assume that there exists a constant $C$ such that 
  for every $x\in \re^n$ and $0<r<1$,
\begin{equation}\label{eq:density}
\mu(B(x,r)) \leq C r^{n-2} h(r),
\end{equation}
where $h$ is a nonnegative increasing function satisfying  $\int_0 \frac{h(r)}{r^2}dr<\infty$.

Then $u$ is $C^{1+\psi}$, with 
\begin{equation}\label{eq:modulus}
\psi (r) = \int_0^{2r} \frac{h(s)}{s^2}ds + r\int_r^1 \frac{h(s)}{s^3}ds.
\end{equation}
\end{prop}

\begin{proof} When $n=2$ the result  follows from \cite[\S III.4]{garnett}. We will also need it for $n=4$, 
so let us  indicate how to adapt the proof to this case.
 
Since the problem is local, we can assume that $u$ is harmonic outside $B(0, 1/2)$, and by using the  Riesz decomposition, 
it is enough to prove the result when $u$ is the canonical solution of the Laplace equation, that is 
$$u(x)= \int \frac{d\mu(y)}{\norm{x-y}^2}.$$ Taking (at least formally) the derivative with respect to $x_j$ ( $x= (x_1, \ldots, x_n)$), we get 
$\frac{\fr u }{\fr x_j} = K_j *\mu$ where $K_j(z) = -{z_j}/{\norm{z}^4} = O(\norm{z}^{-3})$. Conversely, if for every $j$,
$K_j *\mu$ is a  continuous function, then $u$ is indeed $C^1$ and the formula  $\frac{\fr u }{\fr x_j} = K_j *\mu$ holds.

Now we set  $r= 10 \norm{x-x'}$ and write
$$\frac{\fr u }{\fr x_j}(x) - \frac{\fr u }{\fr x_j}(x')= \int_{B(x,r)} (K_j(x-y) -   K_j(x'-y))d\mu(y)  + \int_{B(x,r)^c}(K_j(x-y) -   K_j(x'-y))  d\mu(y).$$ 
 
To estimate 
the first term in this equality, we notice that $B(x, r ) \subset B(x', \frac{11}{10} r)$ 
(hence the $2r$ in the first integral of \eqref{eq:modulus}), so it is enough to estimate
$\int_{B(x,r)} K_j(x-y)  d\mu(y)$.
We have that
\begin{align*}
 \abs{\int_{B(x,r)} K_j(x-y)  d\mu(y)} \leq  \int_{B(0,r)} \frac{d\mu(x+z)}{\norm{z}^3} &= \int_{1/r}^\infty \mu(B(0, \unsur{t}))3t^2 dt \\
&\leq 
3\int_{1/r}^\infty h\left(\unsur{t}\right)dt = 3\int_0^r \frac{h(s)}{s^2}ds.
\end{align*} where the  equality on the first line follows from the formula $\int f d\mu = \int_0^\infty \mu(\set{f>t^3}) 3t^2dt$.

For the second term we use the fact that the partial derivatives of $K_j$ are $O(\norm{z}^{-4})$ so that when 
$\norm{x-y}\geq r$ (whence $\norm{x'-y}\geq 9r/10$) we have $\abs{K_j(x-y) -   K_j(x'-y)}\leq C\frac{\norm{x-x'}}{\norm{x-y}^4}$. 
As above we  infer that 
\begin{align*}
\abs{\int_{B(x,r)^c}(K_j(x-y) -   K_j(x'-y))  d\mu(y)} &\leq C r\int_{B(0,1)\setminus B(0,r)} \frac{d\mu(x+z)}{\norm{z}^4}  \\
&= C r\int_{1}^{1/r} \mu(B(0, \unsur{t}))4t^3 dt 
= 4C r\int_r^1 \frac{h(s)}{s^3}ds, 
\end{align*}
which, together with the previous estimate, concludes the proof.
\end{proof}

Direct computation shows the following:

\begin{cor}
With notation as in Proposition \ref{prop:modulus}, if $0<\alpha<1$ and  $h(r)=O(r^{1+\alpha})$, then $u$ is $C^{1,\alpha}$. 
If $h(r) = O(r^2)$ then $u$ is $C^{1+\psi}$ with $\psi(r) = r\abs{\log r}$.   
\end{cor}

Later on (see \S \ref{subs:hausdorff}) we shall see that for the currents constructed in Section 
\ref{sec:continuous} we always have $h(r)/r^2\cv\infty$ so the potentials are 
 less regular than  $C^{1+r\abs{\log r}}$. It is not a surprise  that if 
  $h(r)/r^2$ diverges slowly enough, then 
 every regularity below  $C^{1+r\abs{\log r}}$ can be reached.
 This is the contents of the next result, which  follows from elementary calculus.

\begin{prop}\label{prop:anyweight}
 Let $\psi$ be a gauge function such that 
  $\frac{\psi(r)}{r\abs{\log r}}\cv\infty$ as $r\cv 0$. 
Then there exists a decreasing function $\theta$ such that $r\mapsto h(r) = r^2\theta(r)$ satisfies the assumptions of Proposition \ref{prop:modulus} and 
\begin{equation}\label{eq:opsi}
 \int_0^{2r} \frac{h(s)}{s^2}ds + r\int_r^1 \frac{h(s)}{s^3}ds =O(\psi(r)).
\end{equation}
\end{prop}

It will follow from the proof that we can further assume that $\frac{\theta'}{\theta} = o\left(\unsur{r\log r}\right)$. 
Let us  study this case first.

\begin{lem}\label{lem:specialcase}
Let $\theta$ be a function defined in a neighborhood of $0\in \re^+$. Assume that $\theta$ is $C^1$, decreasing, $\lim_{0^+}\theta = +\infty$, and 
 $\frac{\theta'}{\theta} = o\left(\unsur{r\log r}\right)$. 

Then with notation as in Proposition \ref{prop:modulus},  if $h(r)= r^2\theta(r)$, then $\psi(r) = O( r\abs{\log r} \theta(r))$. 
\end{lem}

The assumption of the lemma
 holds e.g. when $\theta(r)=   {\log}\circ \log\circ  \cdots \circ \abs{\log r}$. 
 For the limiting case  $\theta(r) = \abs{\log r}$ (for which 
$\frac{\theta'}{\theta}=\unsur{r\log r}$)
 the assumption does not hold but the reader may check that the conclusion is still valid.

\begin{proof} This is very elementary. Notice first that the assumption on $\theta$ implies that  $\theta(r) = o(\log r)$. Notice also that since $\theta$ is decreasing, $\theta(2r)\leq \theta(r)$.

Consider now the first integral in \eqref{eq:modulus}. Integrating by parts yields
$$\int_0^{2r} \frac{h(s)}{s^2} ds = \int_0^{2r} \theta(s)ds = 2r\theta(2r) - \int_0^{2r} s\theta'(s)ds\sim 2r\theta(2r), $$
 because $s\theta'(s) = o(\theta(s))$. 
  For the second one, integrate by parts again
 $$\int_r^1  \frac{h(s)}{s^3} ds=\int_r^1 \frac{\theta(s)}{s}ds = -\log r \theta(r) -\int_r^1  
 \theta'(s)\log s ds\sim\abs{\log r}\theta(r),$$ for $ \theta'(s)\log s = o\left(\frac{\theta(s)}{s}\right)$. 
We conclude that $\psi(r) \sim r\abs{\log r}\theta(r)$.
  \end{proof}

\begin{proof}[Proof of Proposition \ref{prop:anyweight}]
Let   $\theta_0$ be defined on $(0,r_0)$ by $\theta_0(r) = \frac{\psi(r)}{r\abs{\log r}}$. Replace first
$\theta_0$ with any decreasing function $\theta_1\leq \theta_0$ with $\lim_0\theta_1 = +\infty$. The next step is to 
  replace $\theta_1$ with a function $\theta_2\leq \theta_1$ satisfying the assumptions of Lemma \ref{lem:specialcase}. 
Then   the choice $\theta = \theta_2$ will have the desired properties.
Indeed, put $h(r) = r^2\theta_2(r)$. The assumption on the derivative of $\theta_2$ implies that $h$ 
is increasing near 0, and as we have  
seen, $\theta_2(r) = o(\abs{\log r})$ thus $\int_0\frac{h(r)}{r^2}dr<\infty$. 
 By Lemma \ref{lem:specialcase},  we have 
  $$\int_0^{2r} \frac{h_2(s)}{s^2}ds + r\int_r^1 \frac{h_2(s)}{s^3}ds = O( r\abs{\log r} \theta_2(r)) ,$$
which is an $O(\psi(r))$ be definition of $\theta_2$. 

It remains to see why such a $\theta_2$ exists. By making the change of variables $x=1/r$ we are claiming that for any  function 
$F$ on $\re^+$ increasing to $+\infty$, there exists $G\leq F$ increasing to infinity, and 
such that moreover $\frac{G'}{G} = o\left(\frac{1}{x\log x}\right)$. Put $f= \log F$ and $g=\log G$ so that the requirement is that 
$g'= o\left(\frac{1}{x\log x}\right)$ . 
We construct $g$ as follows. Fix $x_1 \in \re^+$ such that $f(x_1)>0$, put $g(x_1)=f(x_1)$ and declare that $g$ is constant until $x_2$, where $x_2$ is such that $f(x_2)= 2f(x_1)$. From $x_2$ let $g(x) = \log \log\log x - \log \log\log x_2 + g(x_2)$. If  $g\leq f$ forever we are done. Otherwise let $y_1>x_2$ be the least number such that $g(y_1) = f(y_1)$ and repeat the above procedure. 
It is clear that $g\leq f$, $g$ increases to infinity, and $g'(x) = o\left(\frac{1}{x\log x}\right)$.  
\end{proof}

\subsection{Geometry of the vertical slices near the boundary}\label{subs:slice}

We return to the setting of Sections \ref{sec:wermer} and \ref{sec:continuous}, and assume that $(r_n)$ satisfies the hypothesis 
\eqref{eq:rn} of Theorem \ref{thm:wermer_continuous}. Our purpose is now to fix the sequence $(m_n)$.

Throughout this subsection, we work in a fixed  ``vertical'' slice near the boundary. By this,  we mean
a line of the form $\pi^{-1}(z_0)$ where $\pi$ is of the form 
$(z,w)\mapsto z+ \gamma w$, with $\abs{\gamma}\leq \unsur{100}$ and  $\frac{2}{5}\leq z_0\leq \frac{4}{10}$. The choice of $z_0$ 
and $\gamma$ ensures that
$\pi^{-1}(z_0)$  is a vertical graph  in $\set{\frac{3}{8}<\abs{z}<\unsur{2}}\times\dd$. Then by Lemma \ref{lem:construction} {\em (iii.)},
 each variety $\set{P_{n,s} = \alpha}$, $\abs{\alpha}<2\delta_n$ intersects it in exactly $2^n$ points. 
Indeed $\pi^{-1}(z_0)$ is actually a vertical graph in a thin bidisk of the form $D(z_0,r)\times \dd$, in which $\set{P_{n,s} = \alpha}$ is the union of $2^n$ disjoint graphs.

By $X_{n, s}$, $X$, etc. we mean  the trace of these subsets 
on the slice $\pi^{-1}(z_0)$, and we denote by $\mu_n$ (resp. $\mu$) the Laplacian of $u_n$ (resp. $u$)  on that slice, that is, the slice measure of $T_n$ (resp. $T$). 

We use the following notation: $a_n\asymp b_n$ (resp. $a_n\approx b_n$) if there exists $C>0$ independent on $n$ such that $a_n/C\leq b_n\leq C a_n$ (resp. $a_n/C^n \leq b_n\leq C^n a_n$. 

We aim at  proving the following result.

\begin{prop}\label{prop:slice}
 Let $h$ be any increasing function defined in a neighborhood of $0\in \re^+$ such that $h(r)/r^2\cv +\infty$. Then there exists a sequence $(m_n)$ 
such that for every $p$ in the slice and every $r>0$, $\mu(B(p, r))\leq C h(r)$, where $C$ is a universal constant (in particular independent on the slice). 
\end{prop}

By Propositions \ref{prop:modulus} and \ref{prop:anyweight}, this implies that  along the slice, $u$ can be made $C^{1+\psi}$ for an arbitrary gauge $\psi$ satisfying 
 $\frac{\psi(r)}{r\abs{\log r}}\cv\infty$ as $r\cv 0$. In the next section, we 
 extend this regularity  to the bidisk. 

The idea of the proof is to study the geometry of the Cantor set $X$, and the distribution of 
$\mu$ on $X$, by using some techniques from plane conformal geometry. Recall that if $f$ is a univalent mapping defined in a 
topological disk $\Delta\subset\cc$,  the {\em distortion} of  $f$ is defined as $\sup_{z,w\in \Delta} \frac{f'(z)}{f'(w)}$. If $f:\Delta \cv \dd$ is a conformal map with $\abs{f'(0)}=1/R$ we say that the {\em conformal radius} of $\Delta$ is $R$. 

Let us define $R_n= \prod_1^n r_k$ and $M_n = \prod_1^n m_k$. We first describe the basic geometry of $X$. 

\begin{prop}\label{prop:distortion}
 For every $n\in \nn$ and $s\in \mathcal{S}_n$, each component  
of $X_{n,s} = \set{P_{n,s} < \delta_n}$, is up to uniformly bounded distortion, a disk of conformal radius $\approx \frac{R_n}{M_n}$. 

More specifically, if $\Delta$ is such a component,  then
$P_{n,s}:\Delta\cv D(0, \delta_n)$ is a univalent mapping of uniformly bounded distortion, and its derivative is $\approx\delta_n \frac{M_n}{R_n}$.
\end{prop}

\begin{proof}
 We know from Lemma \ref{lem:construction} 
that for every $\alpha \in D(0,2\delta_n)$ the equation $P_{n,s} = \alpha$ has exactly $2^n$ solutions. Since the
 solutions do not collide, they vary holomorphically, and it follows that  $P_{n,s}^{-1} (D(0,2\delta_n))$ 
is the union of $2^n$ topological disks, and $P_{n,s}$ is univalent on each of them. By the Koebe Distortion Theorem, the  
distortion of $P_{n,s}^{-1}\rest{D(0,\delta_n)}$ is bounded by a universal constant, hence the same holds for $P_{n,s}$ on $\Delta$.

What remains to do is to estimate the derivative of $P_{n,s}$ on $\Delta$. We do it by induction,  taking $w$ 
as coordinate on the slice and simply denoting the derivative of $P$  by  $P'$. 
Recall that if $s'=(s,\sigma)$, 
$P_{n+1,s'} = (P_{n,s}-\sigma)^2+\e_{n+1} A_{n+1}$.
It is clear that $\abs{A_n'}\leq 1$. 
By definition of the sets $X_{n,s}$ we have $\abs{P_{n,s}-\sigma} <\frac{\delta_n}{m_{n+1}}$ on $X_{n+1, s'}$. 
On the other hand, since on the slice we have $\abs{A_{n+1}}>\unsur{10}$ (see 
\eqref{eq:an}), we infer that when    $\abs{P_{n+1,s'}}<\delta_{n+1}$ we have 
$$\abs{P_{n,s}-\sigma}^2 \geq \e_{n+1} \abs{A_{n+1}} - \delta_{n+1} \geq 
\frac{\delta_n^2}{m_{n+1}^2}\left(\unsur{20} - \frac{r_{n+1}}{4}\right) \geq \frac{\delta_n^2}{40 m_{n+1}^2}.$$
We conclude that on $X_{n+1,s}$,  
\begin{equation}\label{eq:Xns}
  \unsur{\sqrt{40}}\frac{\delta_n}{m_{n+1}} \leq \abs{P_{n,s}-\sigma} < \frac{\delta_n}{m_{n+1}}.
\end{equation}

Let us first imagine for simplicity  that for all $n$, $P_{n+1,s'}= (P_{n, s}-\sigma) ^2$. In this case we would get that 
$\abs{P'_{n+1,s'}}=2\abs{P_{n,s}'}\abs{P_{n, s}-\sigma}  \asymp \frac{\delta_n}{m_{n+1}}\abs{P_{n,s}'}$ 
on $X_{n+1,s'}$. By  induction this implies that 
$\abs{P'_{n+1,s'}} \approx \prod_0^n \frac{\delta_k}{m_{k+1}}$. An immediate  computation shows that $\delta_{n+1} = 
\frac{R_{n+1}}{4^{n+1}M_{n+1}^2} \prod_0^n \delta_k$ so we conclude that $\abs{P'_{n+1,s'}} \approx \delta_{n+1} 
\frac{M_{n+1}}{R_{n+1}}$.

Now we need to take care of the extra term in $P_{n+1,s'}$. Let $D_n =  \prod_0^{n-1} \frac{\delta_k}{m_{k+1}}= 4^n\delta_n\frac{M_n}{R_n}$. 
We want to prove by induction  that on $X_{n,s}$, $\abs{P_{n,s}'}\approx D_n$. Assume a constant $C$ has been found such that $ C^{-n}D_n\leq 
\abs{P_{n,s}'}\leq C^n D_n$.
By definition of $P_{n+1, s'}$, $\e_{n+1}$, $\delta_{n+1}$  and using \eqref{eq:Xns}, we have 
\begin{align*}
\frac{\big|P_{n+1, s'}'\big|}{D_{n+1}} &\leq 2 \frac{\abs{P_{n, s}'}}{D_{n}} \frac{\big|P_{n,s}-\sigma\big|}{\delta_n/m_{n+1}} + 
\frac{\e_{n+1}}{D_{n+1}} \\ &\leq 2C^n + \frac{\delta_{n+1}r_{n+1}/2}{\delta_{n+1}4^{n+1} M_{n+1}/ R_{n+1} } 
\leq C^n\left(2+ \frac{R_n}{M_{n+1}}\unsur{2\cdot 4^{n+1}C^n}\right)
\end{align*} which is less than $C^{n+1}$ as soon as $C\geq 3$ because $R_n/M_{n+1}$ is super-exponentially small in $n$. 
The reverse inequality being similar, the result is proved. 
\end{proof}

From now on we refer to components of $X_{n,s}$ as {\em components of depth} $n$. 
Let $[\lrd_n, \urd_n]$ be the interval of variation 
 of the conformal radii of components of depth $n$.
By the previous proposition, $\lrd_n\approx \urd_n\approx \frac{R_n}{M_n}$. To get a good distinction between scales, from now on we assume that 
$(m_n)$ has super-exponential growth, so that for every $C$, if $n$ is large enough, 
$C^{n+1} \frac{R_{n+1}}{M_{n+1}} \leq C^{-n} \frac{R_n}{M_n}$, so in particular 
$\urd_{n+1}< \lrd_n$.

We will also need to take into consideration the size of {\em intermediate components 
of depth $n+1$}, that is, the components of the form $X^{\mathrm{int}}_{n+1, s, \sigma}$. By construction, the conformal radius of such a component is $\unsur{m_{n+1}}$ times the radius of the component of depth $n$ in which it sits, so that if $[\lrd^{\mathrm{int}}_{n+1}, \urd^{\mathrm{int}}_{n+1}]$ denotes the corresponding range of radii, we have that $[\lrd^{\mathrm{int}}_{n+1}, \urd^{\mathrm{int}}_{n+1}] = \unsur{m_{n+1}}[\lrd_{n}, \urd_{n}]$. Notice also that the largest component of depth $n+1$ lies in some intermediate component, so $\urd_{n+1}\leq \urd^{\mathrm{int}}_{n+1}$.
So, still assuming that $m_n$ has super-exponential growth we conclude 
that
$\urd_{n+1}< \urd^{\mathrm{int}}_{n+1}< \lrd_n$.
 
We are now in position to estimate $\mu(B(p,r))$ from above for every $p$.

\begin{prop} \label{prop:mass}
Assume that $(m_n)$ has super-exponential growth. There exists a constant  $C$ such that 
for every $p$ in the slice and $n$ large enough the following holds: 
\begin{itemize}
 \item if  $\lrd_{n+1}\leq r \leq \urd^{\mathrm{int}}_{n+1}$, then $\mu(B(p,r))\leq \frac{C^{n}}{M_{n+1}^2}$;
 \item if  $\urd^{\mathrm{int}}_{n+1}\leq r\leq \lrd_{n}$ then $\mu(B(p,r))\leq \frac{C^{n}}{R_n^2} r^2$.
\end{itemize}
\end{prop}

\begin{proof}
Fix $A$ such that 
 $A^{-n} \frac{R_n}{M_n}\leq \lrd_n \leq  \urd_n \leq  A^{n} \frac{R_n}{M_n}$ for all  $n$.
Observe  first that by construction, the $\mu$-mass of a component of depth $n$ equals its $\mu_n$ mass. In particular the mass of a component of depth $n$ is 
$\unsur{2^n\#\mathcal{S}_n}\approx \unsur{M_n^2}$ and the mass of an intermediate component of depth $n+1$ is $\frac{2}{2^{n+1}\#\mathcal{S}_{n+1}}\approx \unsur{M_{n+1}^2}$. 

The argument for estimating the mass of balls is the same in both cases. 
Due to the  bound on distortion, if $\Delta$ is an intermediate component of depth $n+1$, then 
$\mathrm{Diameter}(\Delta)\leq K\urd_{n+1}^{\mathrm{int}}$ and $\mathrm{Area}(\Delta)\geq \unsur{K}\lrd_{n+1}^2$ for some $K$.  
If $\lrd_{n+1} \leq r \leq  \urd^{\mathrm{int}}_{n+1}$,
any intermediate component of depth $n+1$ intersecting $B(p,r)$  must be contained in 
$B(p, (K+1) \urd^{\mathrm{int}}_{n+1})$, which in turn contains 
at most $\frac{\pi (K+1)^2(\urd^{\mathrm{int}}_{n+1})^2}{(\lrd^{\mathrm{int}}_{n+1})^2/K}\leq \pi K (K+1)^2 A^{4n}$ of them, due to the area bound. Hence we conclude that 
$\mu(B(p,r))\leq \frac{C^{n}}{M_{n+1}^2}$ for some $C$.

In the other case, we argue that  an intermediate component of depth $n+1$ intersecting $B(p,r)$  is contained in 
$B(p, r+ K \urd^{\mathrm{int}}_{n+1})\subset B(p, (K+1)r)$, so the total number  of those does not exceed  $\frac{\pi (K+1)^2r^2}{(\lrd^{\mathrm{int}}_{n+1})^2/K}$, and we conclude by using the fact that 
$\lrd^{\mathrm{int}}_{n+1}\approx \frac{R_n}{M_{n+1}}$.
\end{proof}

\begin{proof}[Proof of Proposition \ref{prop:slice}]
 As before, write $h(r) = r^2 \theta(r)$, with $\lim_0\theta= +\infty$. As in Proposition \ref{prop:anyweight} 
we can always replace $\theta$ with some {\em decreasing} function of slower growth, and prove the result for the new $\theta$.
We want to choose $(m_n)$ so that $\mu(B(p,r))\leq r^2\theta(r)$ for small enough $r$.   

By Proposition \ref{prop:mass}, 
when $\lrd_{n+1}\leq r \leq \urd^{\mathrm{int}}_{n+1}$, $\mu(B(p,r)) \leq \frac{C^n}{M_{n+1}^2}$, and $h(r)\geq h(\lrd_{n+1})$, so
  a  sufficient condition for $\mu(B(p,r))\leq h(r)$ is  that 
$$\frac{C^n}{M_{n+1}^2}\leq h\left( A^{-(n+1)}\frac{R_{n+1}}{M_{n+1}}\right),$$ 
where $A$ is as in the proof of Proposition \ref{prop:mass}. This rephrases as 
\begin{equation}\label{eq:theta_1} \theta\left( A^{-(n+1)}\frac{R_{n+1}}{M_{n+1}}\right) \geq C^nA^{n+1}\unsur{R_{n+1}^2}.
\end{equation}

In the alternate case where $ \urd^{\mathrm{int}}_{n+1}\leq r\leq \lrd_n$, since $\theta$ is decreasing, 
if $\theta(\lrd_n) \geq \frac{C^n}{R_n^2} $ we infer that 
 $$\mu(B(p,r))\leq \frac{C^{n}}{R_n^2} r^2 \leq r^2 \theta(\lrd_n) \leq r^2\theta(r).$$ For this a  sufficient condition is  that 
\begin{equation}\label{eq:theta_2}
\theta\left(A^{n}\frac{R_n}{M_n}\right)\geq C^n\frac{1}{R_n^2}. 
\end{equation}

From  \eqref{eq:theta_1} and \eqref{eq:theta_2}, we conclude that to achieve the desired conclusion it is enough that for large $n$, $\theta\left(A^{n}\frac{R_n}{M_n}\right) \geq (CA)^n\unsur{R_n^2}$. Now, since $\lim_0 \theta= +\infty$, it is clear by induction on $n$
 that this condition will be satisfied if $m_n$ is chosen to be sufficiently large. 
\end{proof}

\subsection{Transfer of regularity and conclusion}\label{subs:mongeampere}
To study the regularity of $T$ throughout the bidisk, we  use some basic estimates for solutions of homogeneous 
complex Monge-Ampère equations. Let us introduce some notation from \cite{bt}. Let $u$ be a psh function in some open set $\om$, $\zeta\in \cd$ be a unitary vector, and $r >0$. If $p\in \om_r = \set{p\in \om, \ \dist(p,\fr\om) >r}$, we let 
$$\left(T_{\zeta, r}u\right)(p) = r^{-2}\left( u_{\zeta, r}(p) - u(p)\right) \text{, with } u_{\zeta, r}(p)=\unsur{2\pi}\int_0^{2\pi} u(p+ r\zeta e^{i\theta}) d\theta .$$
Recall also the classical Jensen formula for a subharmonic function in one variable 
$$\unsur{2\pi}\int_0^{2\pi} u(re^{i\theta}) d\theta - u(0) = \int_0^r \frac{n(t)}{t}dt \text{, where } n(r) =\int_{\set{\abs{z}\leq r}}\Delta u, $$ so that if now $u$ is psh in $\om$ and if we denote by $n_{\zeta, r}(p)$ the mass of $dd^cu$ along the flat disk of radius $r$ in the direction $\zeta$, centered at $p$, we infer that $\left(T_{\zeta, r}u\right)(p)  = r^{-2}\int_0^r \frac{n_{\zeta, t}(p)}{t}dt$.

\medskip

We can now finish the proof of Theorem \ref{thm:refined}. Fix a gauge $\psi$ with $\frac{\psi(r)}{r\abs{\log r} }\cv \infty$. By Proposition \ref{prop:anyweight} (and its proof) there exists a decreasing $\theta$, with $\lim_0\theta = +\infty$ and $\frac{\theta'}{\theta} = 
o\left(\frac{1}{r\abs{\log r}} \right)$, such that $h(r)= r^2 \theta(r)$ satisfies \eqref{eq:opsi}. 
Notice that since $\theta$ is decreasing, $h(Ar) = O(h(r))$ for $A\geq 1$.  
By Proposition \ref{prop:slice}, 
we can choose $(m_n)$ so that for every  slice of the form $\pi^{-1}(z_0)$, with $\pi(z,w) = z+\gamma w$, $\abs{\gamma}\leq \unsur{100}$
  and  $\frac25 <\abs{z_0}<\frac{4}{10}$, the slice measures of $T$ satisfy $\mu(B(p,r))\leq Ch(r)$ for every $p$ and $r>0$. 
Notice that the union of these slices contains the open set $\set{\frac25 <\abs{z}<\frac{4}{10}}\times \dd$.

\medskip

Using the above notation, if $\zeta=(\zeta_1, \zeta_2)$ is a unit vector  in $\cd$, with $\abs{\zeta_1}\leq\unsur{100}\abs{\zeta_2}$, 
we have that for every $p\in \set{\frac25 <\abs{z}<\frac{4}{10}}\times \dd$,  $n_{\zeta, r}(p) = O(h(r))$. 
Thus by Jensen's formula we infer that for every such $p$ and $\zeta$,
$$0\leq (T_{\zeta, r}u)(p) = r^{-2}\int_0^r \frac{n_{\zeta,t}(p) }{t}dt \leq C  r^{-2}\int_0^r t\theta(t)dt \leq  C \theta(r),$$ where the last inequality follows from an integration by parts, as in the proof  of Lemma \ref{lem:specialcase}. 
Notice that if $p$ is close to the horizontal boundary of $D\big(0, \frac{4}{10}\big)\times \dd$,  $(T_{\zeta, r}u)\equiv 0$ since $u$ is pluriharmonic there. 

Now if we let $\om =D\big(0, \frac{4}{10}\big)\times \dd$,   by \cite[Theorem 6.4]{bt}, if $r<\e$  
$$\sup\set{(T_{\zeta, r}u)(p), \ p\in \om} = \sup \set{(T_{\zeta, r}u)(p), \ p\in \om, \ \dist(p,\fr\om) <\e },$$ so we conclude that throughout the bidisk $D\big(0, \frac{4}{10}\big)\times \dd$, the estimate $(T_{\zeta, r}u)(p) \leq  C \theta(r)$ holds. 
Now we use the Jensen formula again and the reverse estimate 
$\int_0^r \frac{n(t)}{t}dt \geq \int_{r/2}^r\geq (\log 2) n(\frac{r}{2})$ and we obtain that for every  vector 
 $\zeta$ close to the vertical  as above, and every 
$p\in D\big(0, \frac{4}{10}\big)\times \dd$,  we have  $n_{\zeta, r}(p) = O(r^2\theta(2r))= O(h(r))$.

\medskip

To apply Proposition \ref{prop:modulus} and conclude that $u$ is $C^{1+\psi}$, we need to control the mass 
of small {\em balls} for $\Delta u$, or equivalently for the trace measure $\sigma_T$ of $T$. Because $T$ is a positive current, 
it is well known that 
controlling slice masses in two  directions gives a control of the trace measure. Indeed, let
$\omega_\cd = idz\wedge d\overline z + idw\wedge d\overline w$ (resp. $\omega_\cc= idz\wedge d\overline z$)  be the standard K\"ahler form of $\cd$ (resp. $\cc$). If $\pi_j:(z,w)\mapsto z+\gamma_jw$, $j=1,2$  are two distinct  projections, there exists a constant $C$ depending on the $\gamma _j$ 
such that $\omega_{\cd}\leq C(\pi_1^*\omega_\cc+ \pi_2^*\omega_\cc)$, so 
$\sigma_T = T\wedge \omega_\cd\leq C\sum_{j=1,2} T\wedge \pi_j^*\omega_\cc$. Finally, since the projection (resp. the fibers) of 
$B(p,r)$ under $\pi_j$ are contained in disks of radius $\leq Cr$, 
by the Slicing Formula we infer that   $(T\wedge \pi_j^*\omega_\cc) (B(p,r)) \leq Cr^2 h(Cr) =O(r^2h(r))$, which 
by Proposition \ref{prop:modulus} and our assumption on $\theta$, implies that $u$ is $C^{1+\psi}$ in 
$D\big(0, \frac{4}{10}\big)\times \dd$. Of course to obtain the same result in  $D\big(0, \frac{1}{2}\big)\times \dd$ it suffices to  consider projections closer to the vertical and an exhaustion argument.
\hfill $\square$

\begin{rmk}
 The arguments developed here incidentally show that if $\om$ is a bounded open set and $u\in \mathcal{C}(\om)$ is a 
solution of the homogeneous Monge-Ampère equation
 which is $C^{1,\alpha}$ near $\fr\om$ ($0<\alpha<1$), 
then it is $C^{1,\alpha}$ everywhere, a consequence of \cite{bt} which doesn't seem to be so well-known. This  uses the following 
classical converse to Proposition \ref{prop:modulus}: if a plane subharmonic function $u$ is $C^{1,\alpha}$, then the mass of a ball of radius $r$ is $O(r^{1+\alpha})$. It is also 
 possible to state a $C^{1+\psi}$ analogue of this result,  with a small loss on $\psi$ in the transfer of regularity.  
\end{rmk}

\section{Miscellaneous concluding remarks}

\subsection{Sibony's example}
We first show that in the construction of Sibony referred to in the introduction, 
the Wermer example $\widehat{X}$ has zero trace measure. This is a consequence of the following observation. 

\begin{prop}\label{prop:sibony}
Let $u$ be a nonnegative $C^{1,1}$ psh function in the unit ball of $\mathbb{C}^2$, and let $T=dd^cu$. Then $\sigma_T(\set{u=0})=0$.
\end{prop}

\begin{proof}
Let $E = \set{u=0}$ and assume that  $\sigma_T(E)>0$. Since $T$ is a current with 
$L^{\infty}_{\rm loc}$ coefficients, $\sigma_T$ (or equivalently, the Laplacian  of $u$)
is absolutely continuous with respect to the
 Lebesgue measure  hence $E$ has positive Lebesgue measure. We will prove that $\Delta u$ vanishes a.e. on $E$, thus 
contradicting the fact that $\sigma_T(E)>0$.

 It is classical that $u$ is twice differentiable  a.e. Let $p\in E$ be such a differentiability  point. 
Since $u$ has a  minimum at $p$, $du_p$ vanishes so by the Taylor formula we infer that 
$u(p+h)-u(p)= u(p+h) = O(\norm{h}^2)$ as $h\cv 0$. 

Let $a_4$ be the volume of the unit ball of $\cd$, so that $\mathrm{Leb}(B(p,r)) = a_4r^4$. 
Almost every $p\in E$ is a density point for the Lebesgue measure, 
that is, at such a $p$, $\frac{\mathrm{Leb}(E\cap B(p,r))}{a_4r^4}\cv 1$ 
when $r\cv 0$. 
If we further assume that $u$ is twice differentiable at $p$, we infer that 
\begin{equation}\label{eq:average}
 \unsur{a_4r^4}\int_{B(p,r)} u  = \unsur{a_4r^4}\int_{B(p,r)\setminus E} u = \frac{\mathrm{Leb}(B(p,r)\setminus E)}{a_4r^4} O(r^2) = o(r^2).
\end{equation}

We conclude by using the  Jensen formula, which implies that if $\Delta u\in L^1_{\rm loc}$, then 
$$\lim_{r\cv 0}\frac{1}{r^2}\left( \unsur{a_4r^4}\int_{B(p,r)} u -u(p)\right) = \frac{1}{12} \Delta u (p)$$ almost everywhere 
 (see e.g. \cite[p. 143]{klimek}), which by \eqref{eq:average} implies that $\Delta u=0$ a.e. on $E$.
\end{proof}
 
\subsection{Hausdorff dimension and Lebesgue measure}\label{subs:hausdorff}
If a positive closed current $T$ with $C^{1,\alpha}$ potential, then the mass of a ball of radius $r$ relative to 
its trace measure is $O(r^{3+\alpha})$. In particular it cannot carry any mass on a set of Hausdorff dimension $<3+\alpha$. From this remark we conclude that the support of the current of Theorem \ref{thm:wermer} has dimension 4. It is of course possible to refine this result in the spirit of Theorem \ref{thm:refined} by using appropriate gauge functions.

\medskip

On the other hand the vertical slices of $X$ near the boundary have zero Lebesgue measure, since (we freely use the results and notation
of \S \ref{subs:slice}) they can be covered by $\approx M_n^2$ boundedly distorted balls of radius $\approx \frac{R_n}{M_n}$. So near the boundary, $X$ has zero Lebesgue measure. In a similar fashion, it is clear from the proof of Proposition \ref{prop:mass} that if $h(r)$ is a function such that $\mu(B(p,r))\leq h(r)$ for all $p$ and $r$ then necessarily $\lim_0\frac{h(r)}{r^2}= \infty$.

In particular our currents are never $C^{1,1}$, at least near the boundary 
--it is very likely that the same is true everywhere in the bidisk, but we couldn't prove it. 

\subsection{Wermer examples without subdivision}\label{subs:ordinary}
Our results give some interesting insights on the geometric properties of ordinary  Wermer examples (that is, without the subdivision step, or equivalently $m_n=1$ for all $n$). 
With notation as in Section \ref{sec:wermer}, $X$ is now defined as the nested
 intersection of the sequence of sets $\set{\abs{P_n}<\delta_n}$, with $P_{n+1} = P_n^2 +\e_{n+1}A_{n+1}$, $\e_{n+1}= \delta_n^2/2$ and 
$\delta_{n+1}= \delta_n^2r_{n+1}/4$. 
By Theorem \ref{thm:wermer_continuous}, if the series $\sum_{n\geq 1} \frac{\abs{\log r_n}}{2^n}$ converges, then the associated $T$ has continuous potential, so $X$ is not pluripolar. 

Conversely, 
the logarithmic capacity of a subset of $\cc$ of the form $\set{\abs{P}\leq \delta}$, where $P$ is a monic polynomial of degree $d$,  equals $\delta^{1/d}$, 
so  the capacity of the vertical fibers  of $X$ equals $\lim \delta_n^{1/2^n}$. Using the inductive
 definition of the $\delta_n$ it is easy to see that $\unsur{2^n} \log \delta_n = \sum_{k=1}^n \frac{\abs{\log r_k}}{{2^k}} + O(1)$, so  
if the series $\sum_{n\geq 1} \frac{\abs{\log r_n}}{2^n}$ diverges, the vertical fibers of $X$ are polar. 
It follows from \cite[Theorem 4.1]{ls} that $X$ is complete pluripolar in this case.

\medskip

By the results of \S \ref{subs:slice} the vertical slices of $X$ near the boundary are covered by $2^n$ boundedly distorted balls of super-exponentially small radius $\approx R_n$. Thus, 
even when $\sum_{n\geq 1} \frac{\abs{\log r_n}}{2^n}$ converges, these slices have Hausdorff dimension 0. 
In particular
the potential of $T$ is never   H\"older continuous in this case. 
On the other hand since $R_n$  can have arbitrary slow super-exponential growth, it can be shown that essentially 
any sub-H\"older modulus of continuity can be reached. 

To obtain H\"older continuous examples (of arbitrary exponent $<1$) without subdividing, one modifies the construction by putting $P_{n+1} = P_n^{d_{n+1}} +\e_{n+1}A_{n+1}$ for a well chosen sequence $d_n\cv\infty$. 
Details will appear elsewhere.

\subsection{Some open questions} We conclude with some questions on the geometric properties of currents with $T\wedge T=0$. 
\begin{itemize}
\itm Is it possible to get  $C^{1,1}$  regularity in Theorem \ref{thm:wermer}?  
 \itm What is the best possible regularity for {\em extremal} such currents? 
In our construction extremality is lost in the subdivision process. 
\itm Does there exist a $C^{1, 1}$ psh function $u$ in  $\mathbb{B}$ with $(dd^cu)^2=0$ and  $\supp(dd^c u)=\mathbb{B}$, which is not harmonic on any  holomorphic disk? 
\itm On the other hand, does laminarity hold in the  $C^2$ case? 
\end{itemize}

\end{document}